\documentclass[reqno,a4paper,12pt]{amsart}
\usepackage{amsmath}
\usepackage{graphicx}
\usepackage[english]{babel}
\usepackage[usenames]{color}
\usepackage{dsfont}
\usepackage{subfigure}

\usepackage{ifpdf}
\ifpdf
\usepackage{hyperref}
\else
\usepackage[hypertex]{hyperref}
\hypersetup{backref, colorlinks=true, bookmarks}
\fi

\usepackage{times}
\usepackage[normalem]{ulem}

\newtheorem{theorem}{Theorem}[section]
\newtheorem{lemma}[theorem]{Lemma}
\newtheorem{definition}[theorem]{Definition}
\newtheorem{corollary}[theorem]{Corollary}
\newtheorem{proposition}[theorem]{Proposition}
\theoremstyle{remark}

\numberwithin{equation}{section}

\newcommand{\R}{\mathbb{R}}

\newcommand{\N}{\mathbb{N}}

\newcommand{\ET}[1]{\mathbb{E}_{\Theta}\left[#1\right]}
\newcommand{\PT}[1]{\mathbb{P}_{\Theta}\left[#1\right]}
\newcommand{\PTs}{\mathbb{P}_{\Theta}}
\newcommand{\Pb}[1]{\mathbb{P}\left[#1\right]}
\newcommand{\E}[1]{\mathbb{E}\left[#1\right]}

\DeclareMathOperator{\one}{\mathds{1}}

\renewcommand{\Re}{\mathrm{Re}}

\newcommand{\Sn}{\mathfrak{S}_n}

\begin{document}
\title[Long cycle of permutations with polynomial cycle weights]{Long cycle of random permutations with polynomially growing cycle weights}
\date{\today}
\author[D. Zeindler]{Dirk Zeindler}
\address{Department of Mathematics and Statistics, Lancaster University, Fylde College, Bailrigg, Lancaster LA1 4YF, United Kingdom}
\email{d.zeindler@lancaster.ac.uk}

\begin{abstract}
We study random permutation of $n$ objects with respect to multiplicative measures with polynomial growing cycle weights. 
We determine in this paper the asymptotic behaviour of the long cycles under this measure and 
also prove that the cumulative cycle numbers converge in the region of the long cycles
to a Poisson process.
\end{abstract}


\maketitle

%---------------------------------------------------------------------------
%\tableofcontents

\section{Introduction}
\label{sec:intro}

Let $\Sn$ be the symmetric group of all permutations on elements
$1,\dots,n$. For any permutation $\sigma\in \Sn$, denote by
$C_m=C_m(\sigma)$ the \emph{cycle counts}, that is, the number of
cycles of length $m=1,\dots,n$ in the cycle decomposition
of~$\sigma$; clearly
\begin{equation}\label{eq:sumC}
C_m\ge0 \quad(m\ge1),\qquad \sum_{m=1}^n m\, C_m=n.
\end{equation}
\subsection{Random permutations}
Classical probability measures studied on $\Sn$ are the uniform measure and the Ewens measure.
The uniform measure is well studied and has a long history (see e.g. the first chapter of \cite{ABT02} for a detailed account with references).
The Ewens measure originally appeared in population genetics, see \cite{Ew72}, but has also various applications through its connection with Kingman's coalescent process, see \cite{Ho87}.
Classical results about uniform and Ewens random permutations include
convergence of joint cycle counts towards independent Poisson random variables in total 
variation distance \cite{ArTa92c} and a central limit theorem for 
cumulative cycle counts \cite{DePi85}.
Futhermore, the longest cycles have order of magnitude $n$
and it was established by Kingman (\cite{Ki77}) and  by Vershik and Shmidt (\cite{ShVe77})
that the vector of renormalized and ordered length of the cycles converges in law to a Poisson-Dirichlet distribution. 

In this paper, we study random permutations with respect to the  probability measure
\begin{definition}
\label{def:Pb_measure}
Let $\Theta = \left(\theta_k  \right)_{k\geq1}$ be given, with $\theta_k\geq0$ for every $k\geq 1$.
We define for $\sigma\in\Sn$ the weighted measures on $\Sn$ as
\begin{align}
  \PT{\sigma}
  :=
  \frac{1}{h_n n!} \prod_{k=1}^n \theta_k^{C_k}
  \label{eq:PTheta_with_partition}
\end{align}
with $h_n = h_n(\Theta)$ a normalization constant and $h_0:=1$.
\end{definition}
The measure $\PTs$ has received a lot of attention in recent years and has been studied by many authors. 
The uniform measure and the Ewens measure are special cases of $\PTs$ and correspond to $\theta_k \equiv 1$ and $\theta_k \equiv \vartheta$ with $\vartheta>0$.
Another well studied case are the weights $\theta_k \sim k^\alpha$, see \cite{ErGr08,CiZe13}.
Further studied weights are for instance $\theta_k = \log^m(k)$ and $\theta_k =\vartheta \one_{\{k\leq n^\beta\}}$, 
see \cite{RoZe18}  and \cite{BeScZe17}. An overview can be found  in \cite{ErUe11}.
The motivation to study the measure $\mathbb{P}_\Theta$ has its origins in  mathematical physics.
Explicitly, it occurred in the context of the Feynman-Kac representation of the dilute Bose gas and 
it has been proposed in connection with the study of the Bose-Einstein condensation (see e.g. \cite{BeUeVe11} and \cite{ErUe11}). 
% The weights considered in the literature are 

We consider in this paper the case $\theta_k \sim k^\alpha$ with $\alpha>0$.
It was proved by Ercolani and Ueltschi \cite{ErUe11} that in this case
a typical cycle has a length of order $n^{\frac{1}{1+\alpha}}$ and converges to a Gamma distribution after suitable normalisation.
Dereich and M\"{o}rters extended this in \cite{DeMo15} to a local limit theorem.
Further, it was shown in \cite{ErUe11} that the expectation of total number of cycles is $\approx n^{\frac{\alpha}{1+\alpha}}$,
which was extended by Maples, Nikeghbali and the author in \cite{MaNiZe11} to a central limit theorem.
Furthermore, Ercolani and Ueltschi proved in \cite{ErUe11} that the component process of the cycle counts converges in distribution to mutually \textit{independent} Poisson random variables $Y_m$:
\begin{align}
(C_1, C_2, \ldots, C_b) \overset{d}{\longrightarrow} (Y_1, Y_2,\ldots, Y_b), \quad \quad \text{as } n \rightarrow \infty,
\label{eq:intro_convergence_fixed_b}
\end{align}
where $b\in\N$ is fix and $(Y_k)_{k=1}^b$ are independent Poisson distributed random variables with $\E{Y_k} =\frac{\theta_k}{k}$.
Storm and the author computed in \cite{StZe14a} the total variation distance between the processes in \eqref{eq:intro_convergence_fixed_b}
and have shown that this total variation distance tends to $0$ if and only if $b = o(n^{\frac{1}{1+\alpha}})$.
Further, it was shown in \cite{CiZe13, ErUe11} that the cumulative cycle counts 
\begin{align} 
 w_n(x):=\sum_{k \geq x n^{\frac{1}{1+\alpha}}} C_k
\end{align}
converges to a limit shape after suitable normalisation.
In view of these results it is clear that one cannot expect cycles of order $n$ as $n\to\infty$. 
The exact behavior of the long cycles in the case $\theta_k \sim k^\alpha$ was unknown for a long time and is the main topic of this paper.

\subsection{Main results}
In this paper, we have two main results. 
In order to state these results, we have to introduce some notations.
Let $v_n$ be the solution of the equation
\begin{align}
%  \sum_{k=1}^\infty \theta_k e^{-kv_n} 
%  =
 \sum_{k=1}^\infty \theta_k e^{-kv_n}
 = n.
 \label{eq:saddle_point_equation}
\end{align}
Since $\theta_k\sim k^{\alpha}$ and $\alpha>0$, 
we immediately see that the sum in \eqref{eq:saddle_point_equation} is monotone increasing and tending to $\infty$
as $v_n\searrow0$. 
Thus the solution of \eqref{eq:saddle_point_equation} exists and is also uniquely determined.
It follows from Lemma~\ref{lemma:polylog_asymp} that 
\begin{align}
 v_n\sim \left( \frac{n}{\Gamma(\alpha+1)} \right)^{-\frac{1}{1+\alpha}}.
\end{align}
Furthermore, we set
\begin{align}
 n^*:= (v_n)^{-1}
 \ \text{ and } \
 \ell_n:= \alpha\log(n^*)+(\alpha-1) \log\left(\alpha\log(n^*)\right).
 \label{eq:def_elln}
\end{align}
We denote by $L_{k}=L_{k}\left(\sigma\right)$ the length of the $k$-th
longest cycle in the cycle decomposition of the permutation $\sigma\in\Sn$, counted with multiplicity.
The first main result of this paper is 
\begin{theorem}
\label{thm:Longest0}
Let $K\in\N$ be given.
We have convergence in distribution of
\[
\frac{1}{n^*}\cdot\left(\widetilde{L}_{1}-n^*\ell_n,\dots,\widetilde{L}_{K}-n^*\ell_n\right)
\stackrel{d}{\longrightarrow}
\left( -\log(E_1), \ldots, -\log\left(\sum_{j=1}^K E_j\right)  \right).
\]
under $\PTs$ as $n\to\infty$, where $(E_j)_{j=1}^K$ is a sequence of iid Exp$(1)$ distributed random variables.
\end{theorem}
This theorem follows almost immediately from the second main result.
\begin{theorem}
\label{thm:Poisson_process}
Define for $y\geq0$
\begin{align}
  P_y^{(n)}
 = 
  \sum_{k\geq x_n}^\infty C_k
  \ \text{ with } \
%  x_n:= n^*(\ell_n- \log(y)),
 x_n:= n^*\big(\ell_n + \min\{-\log(y),\ell_n\}\big).
\label{eq:def_Py}
\end{align}
%
% and $P_0^{(n)}:=0$.
Then the stochastic process $\{ P_y^{(n)}, y\geq0\}$ converges under $\PTs$ as $n\to\infty$
weakly in $\mathcal{D}\left[0,\infty\right)$ to a Poisson process
with parameter $1$, where  $\mathcal{D}\left[0,\infty\right)$ denotes the space of c\`{a}dl\`{a}g-functions.
\end{theorem}
We prove this two theorems by computing the generating functions for the relevant quantities and then extracting their asymptotic behavior with the saddle point method.
This paper is structured as follows: In Section~\ref{sec:Preliminaries}, we introduce the necessary preliminaries and 
then give in Section~\ref{sec:proof_of_Poisson} the proof of Theorem~\ref{thm:Longest0} and of Theorem~\ref{thm:Poisson_process}.

\subsection{Notation}
We call two real sequences $(a_n)_{n\in\N}$ and $(b_n)_{n\in\N}$ asymptotically 
equivalent if $\lim_{n\to\infty} a_n/b_n = 1$ and write $a_n \sim b_n$.
Further, we write $a_n\approx b_n$ when there exist constants $c_1,c_2>0$ such that
\begin{align}
 c_1 b_n \leq a_n \leq c_2 b_n
\end{align}
for large $n$. We also use the usual $O$ and $o$ notation,
i.e. $f(n) = O(g(n))$ means that there exists some constant 
$c > 0$ so that $|f(n)| \leq c |{g(n)}|$ for large $n$,
while $f(n) = o(g(n))$ means that for all $c>0$ there exists 
$n_c \in \N$ so that the inequality $|f(n)| \leq c |g(n)|$
holds for all $n > n_c$. 

\section{Preliminaries}
\label{sec:Preliminaries}

We introduce in this section some notation and preliminary results.

\subsection{Generating functions}
\label{sec:gfs}

The (ordinary) generating function of a sequence $(g_k)_{k\geq 0}$ of complex numbers
is defined as the formal power series
\begin{align}\label{eq:G}
g(t): = \sum_{j=0}^\infty g_k t^k.
\end{align}
As usual, we define the
\emph{extraction symbol} $[t^k]\, g(t):= g_k\label{def:ext}$, that is, as the
coefficient of $t^k$ in the power series expansion \eqref{eq:G}
of~$g(t)$.

A generating function that plays an important role in this paper is 
\begin{align}
 g_\Theta(t)
:=
\sum_{k\geq 1}\frac{\theta_k}{k}t^k.
\label{eq:def_g_theta}
\end{align}
The following well-known identity is a special case of the general \emph{P\'olya's
Enumeration Theorem} \cite[p.\,17]{Po37} and is the main tool in this paper to
obtain generating functions. 
\begin{lemma}
\label{lem:Polya}
Let $(a_m)_{m\in \N}$ be a sequence of complex numbers. We then
have as formal power series in $t$
\begin{align}
\sum_{n\in \N}\frac{t^n}{n!}\sum_{\sigma\in \mathfrak S_n}\prod_{j=1}^n a_j^{C_j}
% =
% \sum_{n\in \N} t^n \slan  \prod_{k=1}^\infty a_k^{C_k}
=
\exp\left(\sum_{k\geq 1} \frac{a_k}{k}t^k\right).
\label{eq:lemma:Polya}
\end{align}
If one series in \eqref{eq:lemma:Polya} converges absolutely, so do the other.
\end{lemma}
We omit the proof of this lemma, but details can be found for instance in \cite[p.~5]{Mac95}.
Applying this equation to the normalisation constant $h_n$ in Definition~\ref{def:Pb_measure}, we immediately get
\begin{align}
 \sum_{n=0}^\infty h_n t^n
 =  
 \exp\left(\sum_{k=1}^{\infty}\frac{\theta_k}{k}t^k\right)
 =
 \exp\left( g_\Theta(t)  \right),
\end{align}
with the convention $h_0:=1$.
Equation \eqref{eq:lemma:Polya} can be reformulated as
\begin{align}
\frac{1}{n!}\sum_{\sigma\in\Sn}\prod_{j=1}^{n}a_{j}^{C_j}
=
\left[t^n\right]
\exp\left(\sum_{k=1}^{\infty}\frac{a_k}{k}t^k\right).
\label{eq:relation to perms}
\end{align}
With this formulation, the parameters $(a_j)_{j\in\N}$ can depend on the system size $n$. 
In particular, we have for all $s\in\R$ and all integer sequences $(x_n)_{n\in\N}$
\begin{align}
 h_n \E{e^{s\sum_{k\geq x_n  }^{\infty} C_k}}
 =
  [t^n] \exp\left( (e^s-1) \sum_{k\geq x_n}^{\infty} \frac{\theta_k}{k} t^k + g_\Theta(t)\right).
\label{eq:moment_sum_Cm}
\end{align}

\subsection{Saddle point method}
\label{subsec:log-n-adm}

The formulation in \eqref{eq:moment_sum_Cm} has now the advantage that we can compute the expectation on the LHS of \eqref{eq:moment_sum_Cm}
by extracting series coefficients from the RHS of \eqref{eq:moment_sum_Cm} with complex analysis.
A way to extract these series coefficients is the saddle point method,  a standard tool in asymptotic analysis. 
The basic idea is to rewrite a expression like \eqref{eq:moment_sum_Cm} as a complex contour integral and choose the path of integration in a convenient way. 
The details of this procedure depend on the situation at hand and need to be done on a case by case basis. 
A general overview over the saddle-point method can be found 
in \cite[page~551]{FlSe09}. 
To apply the saddle point method in this paper, we introduce the following definition.
\begin{definition} 
\label{def:admissible}
Let $\bigl( g_n(t)\bigr)_{n\in\N}$ with $g_n(t) = \sum_{k=0}^\infty g_{k,n} t^k$ be given with radius of convergence $\rho > 0$ and $g_{k,n} \geq 0$.
We say that $\bigl( g_n(t)\bigr)_{n\in\N}$ is \emph{$\log$-$n$-admissible} 
if there exist functions $a_n,b_n:[0,\rho) \to \R^+$, $R_n : [0,\rho) \times (-\pi/2, \pi/2) \to \R^+$ and a sequence $(\delta_n)_{n\in\N}$ s. t.
\begin{description}
   \item[Divergence] $b_n(r_n) \to \infty$ and $\delta_n \to 0$ as $n \to \infty$.
   \item[Saddle-point] For each $n$ there exists $r_n \in [0,\rho)$ with 
        \begin{equation}
            a_n(r_n)= n +o\left(\sqrt{b_n(r_n)}\right).
            \label{eq:saddle_point_equation1} 
        \end{equation}
   \item[Approximation] For all $|\varphi| \leq \delta_n$ we have the expansion
        \begin{align}
          g_n(r_ne^{i \varphi})
          =
          g_n(r_n) + i \varphi a_n(r_n)-\frac{\varphi^2}{2} b_n(r_n)
          + R_n(r_n,\varphi)
          \label{eq_admissible_expansion}
        \end{align}
        where $R_n(r_n,\varphi) = o(\varphi^3 \delta_n^{-3})$. % and the implied constant is uniform.
  \item[Width of convergence] We have $\delta_n^2 b_n(r_n) - \log b_n(r_n) \to +\infty$ as $n \to + \infty$.
  \item[Monotonicity] For all $|\varphi| > \delta_n$, we have
        \begin{align}\label{eq:monotonicity}
            \Re\left(g_n(r_n e^{i \varphi})\right) \leq \Re\left(g(r_n e^{\pm i \delta_n})\right).
        \end{align}
\end{description}
\end{definition}

The approximation condition allows us to compute the functions $a_n$ and $b_n$ exactly.  We have
\begin{align}
\label{eq:a_n_general_explicit}
  a_n(r) &= rg_n'(r), \\
  \label{eq:b_n_general_explicit}
  b_n(r) &= rg_n'(r) +r^2 g_n''(r).
\end{align}
Clearly $a_n$ and $b_n$ are strictly increasing real analytic functions in $[0, \rho)$. 
The error in the approximation can similarly be bounded, so that
\begin{align}
  R_n(r,\varphi)= \varphi^3 O\left( a_n(r) + b_n(r) + r^3 g_n'''(r)\right).
\end{align}
We now have
\begin{theorem}
\label{thm:generalasymptotic}
Let $\bigl( g_n(t)\bigr)_{n\in\N}$ be $\log$-$n$-admissible with associated functions $a_n$, $b_n$ and constants $r_n$. 
Call 
$$G_n:=[t^n]\exp\left(g_n(t)\right).$$
Then $G_{n}$ has the asymptotic expansion
\begin{equation}\label{eq:G_n}
 G_{n} = \frac{1}{\sqrt{2 \pi}} (r_{n})^{-n} b_n(r_{n})^{-1/2} \exp\left(g_n(r_{n})\right) (1 + o(1)).
\end{equation}
\end{theorem}
The proof of this theorem can be found in \cite{CiZe13}.
Further
\begin{corollary}
\label{cor:thm:generalasymptotic}
Let $\bigl( g_n(t)\bigr)_{n\in\N}$ be $\log$-$n$-admissible with associated functions $a_n$, $b_n$ and constants $r_n$. 
Let further $\bigl( f_n(t)\bigr)_{n\in\N}$ with $f_n(t) = \sum_{k=0}^\infty f_{k,n} t^k$ be given with radius of convergence $\geq\rho$ and $f_{k,n} \geq 0$.
Then there exists $n_0$ only depending on $\bigl( g_n(t)\bigr)_{n\in\N}$ such that
\begin{align}
 [t^n]f_n(t)\exp\left(g_n(t)\right)
 \leq 2 f_n(r_n)  [t^n]\exp\left(g_n(t)\right)
 \ \text{ for all }n\geq n_0.
\end{align}
\end{corollary}
\begin{proof}
Since $f_{k,n} \geq 0$ and $g_{k,n} \geq 0$, it follows immediately that $[t^n]f_n(t)\exp\left(g_n(t)\right)\geq 0$.
We get with Cauchy's intergral formula and the curve $\gamma(\varphi) = r_n e^{i\varphi}$ that
 \begin{align}
   [t^n]f_n(t)\exp\left(g_n(t)\right)
   &=
   \frac{1}{2\pi i (r_n)^n} \int_{-\pi}^{\pi}f_n(r_n e^{i\varphi}) \exp\left(g_n(r_n e^{i\varphi}) - in\varphi\right) d\varphi\nonumber\\ 
   &\leq
   \frac{1}{2\pi (r_n)^n} \int_{-\pi}^{\pi}\left|f_n(r_n e^{i\varphi}) \exp\left(g_n(r_n e^{i\varphi}) - in\varphi\right)\right| d\varphi\nonumber\\ 
   &\leq 
    \frac{f_n(r_n) }{2\pi (r_n)^n}  
    \int_{-\pi}^{\pi} \exp\left(\Re(g_n(r_n e^{i\varphi}))\right) d\varphi.
   \label{eq:cor_saddle1}
  \end{align}
We now can compute the last integral as in the proof of Theorem~\ref{thm:generalasymptotic}. 
These computations are (almost) identical and we thus omit them. We then get
\begin{align}
 \int_{-\pi}^{\pi} \exp\left(\Re(g_n(r_n e^{i\varphi}))\right) d\varphi
 =
 \sqrt{2 \pi} b_n(r_{n})^{-1/2} \exp\left(g_n(r_{n})\right) (1 + o(1)).
 \label{eq:cor_saddle2}
\end{align}
Note that the $(1 + o(1))$ depends on $\bigl( g_n(t)\bigr)_{n\in\N}$, but not on $\bigl( f_n(t)\bigr)_{n\in\N}$.
Combining \eqref{eq:cor_saddle1} and \eqref{eq:cor_saddle2} with Theorem~\ref{thm:generalasymptotic} then completes the proof. 
\end{proof}

\subsection{Approximation of Sums}
\label{sec:approx_sums}
We require for our argumentation the asymptotic behaviour of the generating function 
$g_\Theta(t)$ as $t$ tends to the radius of convergence, which is $1$ in our case. 
\begin{lemma}
\label{lemma:polylog_asymp}
Let $(v_n)_{n\in \N}$ a sequence of positive numbers with $v_n\downarrow 0$ as $n\to+\infty$. 
We have for all $\delta\in\R\setminus\{-1,\,-2,\,-3,\dots\}$
\begin{align}
\sum_{k=1}^\infty k^\delta \exp(-k v_n)
=
\Gamma(\delta+1) v_n^{-\delta-1} + \zeta(-\delta) +O(v_n).
\end{align}
$\zeta(\cdot)$ indicates the Riemann Zeta function. 
\end{lemma}
This lemma can be proven with Euler Maclaurin summation formula or with the Mellin transformation.
These computations are straightforward and the details of the proof 
with the Mellin transformation can be found for instance in \cite[Chapter~VI.8]{FlSe09}.
We thus omit the proof of this lemma.

We require also the behaviour of partial sum $\sum_{k\geq x_n}^\infty  k^\delta \exp(-k v_n)$ as $x_n\to\infty$ and as $v_n\to 0$. 
We have
\begin{proposition}
\label{prop:eq:eular_mac_partial}
 Let $\delta\in\R$ be given. Let further $(v_n)_{n\in\N}$ and $(x_n)_{n\in\N}$ be sequences with $v_n>0$, $v_n\to0$ and $x_n v_n\to\infty$. 
 We then have
 \begin{align}
\sum_{k=x_n}^\infty k^\delta e^{-k v_n}
=
\int_{x_n}^\infty x^\delta e^{-x v_n} dx 
+ 
x_n^\delta e^{-x_n v_n}\left(\sum _{k=0}^{N} Q_k(1/x_n, v_n)   +O\left(v_n^{-N-1} \right)\right)
\label{eq:eular_mac_partial1}
\end{align}
where $Q_k(\cdot,\cdot)$ are a homogeneous polynomials of degree $k$ with $Q_0=1$. 
Furthermore
\begin{align}
 \int_{x_n}^\infty
 x^\delta \exp(-x v_n) dx
 &=
 \frac{x_n^\delta e^{-x_n v_n}}{v_n} \left(\sum_{j=0}^N \frac{(\delta)_j}{(x_n v_n)^j}+ O(x_n v_n)^{-N-1}  \right),
 \label{eq:eular_mac_partial2}
\end{align}
where $(\delta)_0=1$ and $(\delta)_j = \delta(\delta-1)\cdots(\delta-j+1)$ for $j\geq 1$.
\end{proposition}

\begin{proof}
We first proof \eqref{eq:eular_mac_partial2}. 
We get with $N$ times partial integration
\begin{align}
 \int_{x_n}^\infty
 x^\delta \exp(-x v_n) dx
 &=
 \sum_{j=0}^N (\delta)_j x_n^{\delta-j} v_n^{-j-1} e^{-x_n v_n}
 +
   (\delta)_{N+1} v_n^{-N-2} \int_{x_n}^\infty x^{\delta-N-1} e^{-x v_n} dx \nonumber\\
 &=
 \frac{x_n^\delta e^{-x_n v_n}}{v_n} \sum_{j=0}^N \frac{(\delta)_j}{(x_n v_n)^j}
 +O\left( v_n^{-N-2}  \int_{x_n}^\infty x^{\delta-N-1} e^{-x v_n}\right).
  \label{eq:eular_mac_first_error_term2}
\end{align}
We now can assume that $N>\delta$ and thus
\begin{align}
 \int_{x_n}^\infty x^{\delta-N-1} e^{-x v_n} dx
 \leq 
 x_n^{\delta-N-1} \int_{x_n}^\infty  e^{-x v_n} dx
 =
 \frac{x_n^{\delta-N-1} }{v_n} e^{-x_n v_n}.
\end{align}
Inserting this into \eqref{eq:eular_mac_first_error_term2} completes the proof of \eqref{eq:eular_mac_partial2}.
For the proof of \eqref{eq:eular_mac_partial1}, we use the Euler-Maclaurin summation  formula  and obtain
\begin{align}
\sum_{k=x_n}^\infty k^\delta e^{-k v_n}
=
\int_{x_n}^\infty x^\delta e^{-x v_n} dx 
+ 
\sum _{k=0}^{N+1}\frac{B_{k+1}}{(k+1)!} f^{(k)}(x_n)
+
O\left(\int_{x_n}^\infty |f^{(N+2)}(x)| dx\right),
\end{align}
where  $B_{k}$ is the kth Bernoulli number and $f(x) = x^\delta e^{-x v_n}$ and $f^{(k)}$ the kth derivative of $f$.
A straight forward computation shows that 
\begin{align}
 f^{(k)}(x) = P_k(1/x, v_n) f(x),
\end{align}
where $P_k(\cdot,\cdot)$ is a homogeneous polynomial of degree $k$. 
Thus we have 
\begin{align*}
\sum_{k=x_n}^\infty k^\delta e^{-k v_n}
=
\int_{x_n}^\infty x^\delta e^{-x v_n} dx 
+ 
x_n^\delta e^{-x_n v_n}\left(\sum _{k=0}^{N+1} Q_k(1/x_n, v_n)\right)
+
O\left(\int_{x_n}^\infty |f^{(N+2)}(x)| dx\right),
\end{align*}
with $Q_k(\cdot,\cdot)=\frac{B_{k+1}}{(k+1)!} P_k(\cdot,\cdot)$.
Since $x_n v_n\to\infty$, we have that $1/x_n =o(v_n)$ and thus
\begin{align}
 \int_{x_n}^\infty |f^{(N+2)}(x)| dx
 =
 \int_{x_n}^\infty |P_{N+2}(1/x, v_n)|f(x) dx
 =
 O\left(v_n^{N+2}\int_{x_n}^\infty f(x) dx\right).
 \label{eq:eular_mac_first_error_term}
\end{align}
Using \eqref{eq:eular_mac_first_error_term2} and that $1/x_n =o(v_n)$ completes the proof.
\end{proof}

\section{Proof of the main results}
\label{sec:proof_of_Poisson}
We give in this section the proofs of the Theorems~\ref{thm:Longest0} and \ref{thm:Poisson_process}.
We begin with the proof of Theorem~\ref{thm:Poisson_process}. 
We proceed in two steps.
In Section~\ref{sec:finite_dimensional}, we show that the finite dimensional distributions of the process $ P_y^{(n)}$ converges to the finite dimensional distributions of a Poisson process.
In Section~\ref{sec:Tightness}, we show that the sequence $(P^{(n)})_{n\in\N}$ is tight. 
This then completes the proof of Theorem~\ref{thm:Poisson_process}.
Finally, we use in Section~\ref{sec:proof_longest} the Theorem~\ref{thm:Poisson_process} to prove Theorem~\ref{thm:Longest0}.
\subsection{Finite dimensional distributions}
\label{sec:finite_dimensional}
We first have to show for all $0\leq  y_1<y_2<\cdots<y_K$ that we have
\begin{align}
 \left(P_{y_1}^{(n)},P_{y_2}^{(n)}- P_{y_1}^{(n)}, P_{y_3}^{(n)}- P_{y_2}^{(n)},\ldots, P_{y_K}^{(n)}-P_{y_{K-1}}^{(n)} \right)
 \stackrel{d}{\longrightarrow}
 \left(Y_1,\ldots, Y_K \right),
 \label{eq:finite_dimensional distributions}
\end{align}
where $(Y_k)_{k=1}^K$ is a sequence of independent, Poisson distributed random variables such that $\E{Y_k} =y_k-y_{k-1}$ for all $1\leq k\leq K$. 

We begin with the proof of \eqref{eq:finite_dimensional distributions} for the case $K=1$.
Thus we have to determine the asymptotic behaviour of $P_{y}^{(n)}$ for $y\geq 0$ fix.
We do this by determining the asymptotic behaviour of the moment generating function of $P_{y}^{(n)}$.  
Recall, 
\begin{align*}
  P_y^{(n)}
 = 
  \sum_{k\geq x_n}^\infty C_k
  \ \text{ with } \
 x_n:= n^*\big(\ell_n + \min\{-\log(y),\ell_n\}\big),
\end{align*}
where 
\begin{align}
 n^*= (v_n)^{-1} 
 \ \text{ and } \
 \ell_n=  \alpha\log(n^*)+(\alpha-1) \log\left(\alpha\log(n^*)\right)
\end{align}
and $v_n$ is the solution of the equation $ \sum_{k=1}^\infty \theta_k e^{-kv_n}= n$. 
We now have to distinguish the two cases $y>0$ and $y=0$.
We begin with the case $y>0$. In this case we have for $n$ large enough
\begin{align}
  x_n= n^*\big(\ell_n  -\log(y)\big).
\end{align}
Using \eqref{eq:moment_sum_Cm}, we obtain for $s\in\R$.
\begin{align}
 h_n \E{e^{sP_{y}^{(n)}}}
 =
  [t^n] \exp\left( (e^s-1) \sum_{k\geq x_n} \frac{\theta_k}{k} t^k + g_\Theta(t)\right)
  =: 
  [t^n] \exp\left(  g_{n,s}(t)\right)
 \label{eq:moment_sum_Cm2}
\end{align}
with  $h_n$ as in Definition~\ref{def:Pb_measure}. We now have 
\begin{lemma}
\label{lem_log:admissilbe:finite_dist}
Let $g_{n,s}(t)$ be as in \eqref{eq:moment_sum_Cm2}.
Then sequence $(g_{n,s})_{n\in\N}$ is $\log$-$n$-admissible for all $s\in\R$.
Further, the 
\begin{align}
b_n(r_n) \sim \Gamma(\alpha +2) (n^*)^{\alpha+2},  r_n= e^{-v_n} \text{ and }   \delta_n = (v_n)^{\xi}, 
\end{align}
where $\xi$ is any real number such that $\frac{\alpha+3}{3}< \xi <\frac{\alpha+2}{2}$.
\end{lemma}
\begin{proof}
We begin with the case $\theta_k= k^\alpha$. 
We use \eqref{eq:a_n_general_explicit} and get with Lemma~\ref{lemma:polylog_asymp}, Proposition~\ref{prop:eq:eular_mac_partial} and the definition of $v_n$ that
\begin{align*}
 a_n(r_n)
 &=
 (e^s-1) \sum_{k\geq x_n}^{\infty} k^{\alpha} e^{-k v_n} + \sum_{k=1}^{\infty} k^{\alpha} e^{-k v_n}
 =
 n+  (e^s-1) \frac{x_n^\alpha e^{-x_n v_n}}{v_n}(1+o(1))\\
 &=
  n+  (e^s-1) (n^*)^\alpha(\ell_n- \log(y))^\alpha \cdot y n^* e^{- \ell_n} (1+o(1))  \\
 &=
 n + \alpha y(e^s-1) n^* \log(n^*)   (1+o(1)).
\end{align*}
Similarly, we get 
\begin{align*}
 b_n(r_n)
 &=   \Gamma(\alpha +2) (n^*)^{\alpha+2}(1+o(1)),\\
 R_n(r,\varphi)
 &= 
 \varphi^3 O\left( (n^*)^{\alpha+3}\right).
\end{align*}
It is straight forward to see that replacing $\theta_k= k^\alpha$ by $\theta_k\sim k^\alpha$ in the above computations
has only an influence to the $o(1)$ terms. 
For the $\log$-$n$-admissibility, we have now  to check five conditions. 
For these, we need also that $v_n \sim n^ {-\frac{1}{1+\alpha}} (\Gamma(\alpha+1))^{\frac{1}{1+\alpha}}$.
\begin{description}
 \item[Divergence] We clearly have $\delta_n\to0$ and $b_n(r_n)\to\infty$. Thus this condition is fulfilled. 
 \item[Saddle-point] We require   $a_n(r_n)= n +o\left(\sqrt{b_n(r_n)}\right)$. This condition is clearly fulfilled. 
 \item[Approximation] We need that  $R_n(r,\varphi) =o(\varphi^3 \delta_n^{-3})$. 
            By the definition of $\delta_n$ and $\xi$, we have $\delta_n^{-3} =(n^*)^{3\xi}$ and $3\xi>\alpha+3$.
            Thus $(n^*)^{\alpha+3} = o(\delta_n^{-3})$. Thus this condition is fulfilled. 
  \item[Width of convergence] We have $\delta_n^2 b_n(r_n) \sim \Gamma(\alpha +2) (n^*)^{\alpha+2-2\xi}$. 
       Since $\alpha+2-2\xi>0$, this condition is also fulfilled.
  \item[Monotonicity] The computations for this point are a little bit more involved, but are almost the same as in \cite[Page~25]{CiZe13} and we thus omit it. 
\end{description}
This completes the proof.
\end{proof}
Lemma~\ref{lem_log:admissilbe:finite_dist} shows that we can apply Theorem~\ref{thm:generalasymptotic} to \eqref{eq:moment_sum_Cm2} and
thus compute the asymptotic behaviour of $\E{e^{sP_{y}^{(n)}}}$.
Furthermore, using $s=0$ in \eqref{eq:moment_sum_Cm2}, we see that $g_{n,0}(t) = g_\Theta(t)$.
Thus we can use  Lemma~\ref{lem_log:admissilbe:finite_dist} to compute the asymptotic behaviour of $h_n$.
Note that $r_n$ and the leading term of $b_n(r_n)$ do not depend on $s$.
This implies together with Theorem~\ref{thm:generalasymptotic} that 
\begin{align}
 \E{e^{sP_{y}^{(n)}}}
 =
 \exp\left( (e^s-1) \sum_{k\geq x_n} \frac{\theta_k}{k} (r_n)^k \right) (1+o(1)).
\end{align}
Using again Proposition~\ref{prop:eq:eular_mac_partial}
\begin{align}
 \sum_{k\geq x_n}  k^{\alpha-1} e^{-x_n v_n}
 \sim
 \frac{x_n^{\alpha-1}}{v_n} e^{-x_n v_n} 
 \sim
  (n^*)^{\alpha}  (\alpha  \log(n^*) )^{\alpha -1} \cdot \frac{y (n^*)^{-\alpha } }{(\alpha  \log(n^*) )^{\alpha -1}} = y.
  \label{eq:sum_alpha-1}
\end{align}
We have by assumption $\theta_k\sim k^\alpha$. Thus there exists for all $\epsilon>0$ a $k_0=k_0(\epsilon)$ such that
$(1-\epsilon) k^\alpha \leq \theta_k\leq (1+\epsilon) k^\alpha$ for all $k\geq k_0$.
By definition, we have $x_n\to \infty$ and therefore  we immediately get also
\begin{align}
 \sum_{k\geq x_n}  \frac{\theta_k}{k} e^{-x_n v_n}
 \sim y.
\end{align}
Since $s\in\R$ is fix, this implies that $ \E{e^{sP_{y}^{(n)}}} \to \exp\left( (e^s-1) y \right)$.
This completes the proof for the case $y>0$.
For $y=0$, we have $ x_n= 2n^* \ell_n$ and get as in \eqref{eq:sum_alpha-1}
\begin{align}
 \sum_{k\geq x_n}  k^{\alpha-1} e^{-x_n v_n}
 \sim
   \frac{(n^*)^{\alpha}  (2 \alpha  \log(n^*) )^{\alpha -1} \cdot(n^*)^{-2\alpha } }{(\alpha  \log(n^*) )^{2(\alpha -1)}} 
  =
  O\left(   (n^*)^{-\alpha}  ( \log(n^*) )^{1-\alpha} \right).
  \label{eq:sum_alpha-1.2}
\end{align}
The remaining computations are the same as for $y>0$ and thus we get $P_{0}^{(n)} \stackrel{d}{\to}0$.
This completes the proof of \eqref{eq:finite_dimensional distributions} for the case $K=1$.
For the general case, define 
\begin{align*}
 x_{n,j}:= n^*(\ell_n+ \min\{-\log(y_j),\ell_n\}).
\end{align*}
% 
%  for $1\leq j\leq K$, $x_{n,0} :=\infty$ and $y_0:=0$.
We then have for $s_1,\ldots,s_K\in\R$
\begin{align}
 h_n \E{e^{\sum_{j=1}^K s_j\left(P_{y_j}^{(n)}-P_{y_{j-1}}^{(n)}\right)}}
 =
   [t^n] \exp\left( \sum_{j=1}^L (e^{s_j}-1) \sum_{x_{n,j}\leq k<x_{n,j-1}} \frac{\theta_k}{k} t^k + g_\Theta(t)\right).
 \label{eq:moment_sum_Cm3}
\end{align}
It is now straight forward to see that we can use the exactly same argumentation as for $K=1$.
The only difference is that the notation is more cumbersome.
This completes the proof of \eqref{eq:finite_dimensional distributions}.

\subsection{Tightness}
\label{sec:Tightness}

To complete the proof of Theorem~\ref{thm:Poisson_process}, 
we have to show that the process $\{P_{y}^{(n)}, y\geq 0\}$ is tight.
By \cite[Theorem~13.5 and~(13.14)]{Bi99}, it is sufficient to show for each $M>0$ that 
\begin{align}
 \ET{\left(P_{y}^{(n)}-P_{y_1}^{(n)}\right)^{2}\left(P_{y_2}^{(n)}-P_{y}^{(n)}\right)^{2}}
=
O\left(\left(y_{2}-y_{1}\right)^{2}\right)
\label{eq:tightness_criterium}
\end{align}
uniformly in $y,y_1,y_2$ with  $0\leq y_{1}\leq y\leq y_{2}\leq M$.
Note that we can assume that 
\begin{align*}
%  n^*\left(\log(y_2) -\log(y_1)\right)\geq 1. 
y_2\geq y_1 e^{v_n}.
\end{align*}
Otherwise $\left(P_{y}^{(n)}-P_{y_1}^{(n)}\right)^{2}\left(P_{y_2}^{(n)}-P_{y}^{(n)}\right)^{2} =0$
and the above equation is trivially fulfilled. 
Using \eqref{eq:moment_sum_Cm3}, we immediately get for $s_1,s_2\in\R$
\begin{align*}
 &\ET{e^{s_{1}\left(P_{y}^{(n)}-P_{y_1}^{(n)}\right)+s_{2}\left(P_{y_2}^{(n)}-P_{y}^{(n)}\right)}}\\
 =\,&
 \frac{1}{h_n}\left[z^{n}\right]
 \exp\big(\left(\mathrm{e}^{s_{1}}-1\right)F_{n,y_{1},y}\left(t\right)+\left(\mathrm{e}^{s_{2}}-1\right)F_{n,y,y_{2}}\left(t\right)\big)
%   e^{g_\Theta(t)}
 \exp\left(g_\Theta(t)\right)
\end{align*}
with
\begin{align}
 F_{n,u,v}\left(t\right)
 :=
 \sum_{x_{n,u}\leq k<x_{n,v}}\frac{\theta_k}{k}t^{k}
 \ \text{ for } \ 0\leq u\leq v \leq M 
 \ \text{ and }
\end{align}
where $x_{n,w} := n^*(\ell_n + \min\{-\log(w),\ell_n\})$. 
We now have
\begin{align*}
 & \ET{\left(P_{y}^{(n)}-P_{y_1}^{(n)}\right)^{2}\left(P_{y_2}^{(n)}-P_{y}^{(n)}\right)^{2}}
=  
\left.\frac{\partial^{2}}{\partial s_{2}^{2}}\frac{\partial^{2}}{\partial s_{1}^{2}}
\ET{\mathrm{e}^{s_{1}\left(P_{t}-P_{t_{1}}\right)+s_{2}\left(P_{t_{2}}-P_{t}\right)}}\right|_{s_1=s_2=0}.
\end{align*}
Calculating the derivatives and entering $s_1=s_2=0$ gives
\begin{align*}
  \ET{\left(P_{y}^{(n)}-P_{y_1}^{(n)}\right)^{2}\left(P_{y_2}^{(n)}-P_{y}^{(n)}\right)^{2}}
  =
  \frac{1}{h_n}\left[t^{n}\right]f_{n}\left(t\right)\exp\left(g_\Theta(t)\right)
\end{align*}
with
\begin{align*}
f_{n}\left(t\right):=F_{n,y_{1},y}\left(t\right)\left(1+F_{n,y_{1},y}\left(t\right)\right)F_{n,y,y_{2}}\left(t\right)\left(1+F_{n,y,y_{2}}\left(t\right)\right).
\end{align*}
By the definition, we have $f_n(t) = \sum_{k=0} f_{n,k} t^k$ with all $f_{n,k} \geq 0$.
Furthermore $g_\Theta(t)$ is $\log$-$n$-admissible. 
This follows immediately from Lemma~\ref{lem_log:admissilbe:finite_dist} using $s=0$.
Thus we get with Corollary~\ref{cor:thm:generalasymptotic} that  three exists  a $n_0$ only dependent on $g_\Theta(t)$ such that
\begin{align*}
  \ET{\left(P_{y}^{(n)}-P_{y_1}^{(n)}\right)^{2}\left(P_{y_2}^{(n)}-P_{y}^{(n)}\right)^{2}}
  \leq 
  2 f_{n}\left(r_n\right)
  \ \text{ for all } n\geq n_0.
\end{align*}
We therefore have to estimate $f_{n}\left(r_n\right)$.
We have
\begin{align*}
f_{n}\left(r_n\right)\leq \left(F_{n,y_{1},y_{2}}(r_n)\big(1+F_{n,y_{1},y_{2}}(r_n)\big)\right)^2.
\end{align*}
Since $0\leq y_1<y_2\leq M$, it is sufficient to show that 
\begin{align}
 F_{n,y_{1},y_{2}}(r_n)= 
 \sum_{x_{n,y_2}\leq k<x_{n,y_1}} k^{\alpha-1} e^{-kv_n} 
 \leq 5 (y_2-y_1).
\end{align}
Using that the function $x^{\alpha-1} e^{-x v_n}$ is monotone decreasing for $x \geq (\alpha-1) n^*$ 
and the variable substitution $x=n^*(\ell_n -\log(u))$ gives
\begin{align}
  \sum_{x_{n,y_2}\leq k<x_{n,y_1}} k^{\alpha-1} e^{-kv_n}
  &\leq
  \int_{x_{n,y_2}-1}^{x_{n,y_1}} x^{\alpha-1} e^{-xv_n}dx\nonumber\\
  &= 
   \int_{\max\{y_1,e^{-\ell_n}\}}^{e^{v_n}\max\{y_2,e^{-\ell_n}\}} (n^*)^{\alpha-1}\big(\ell_n -\log(u) \big)^{\alpha-1} e^{-\ell_n} \, du\nonumber\\
  &\leq 
     \frac{1}{(\alpha\log n^*)^{\alpha-1} } \int_{\max\{y_1,e^{-\ell_n}\}}^{e^{v_n}\max\{y_2,e^{-\ell_n}\}} \big(\ell_n -\log(u) \big)^{\alpha-1}  \, du.
\label{eq:thightness_before_alpha}
\end{align}
We have to distinguish the cases  $0<\alpha\leq 1$ and $\alpha>1$.
For $0<\alpha\leq 1$ we use that $0\leq y_1<y_2\leq M$ and that 
\begin{align}
 \frac{\big(\ell_n -\log(M) \big)^{\alpha-1}}{(\alpha\log n^*)^{\alpha-1} } \longrightarrow 1
 \ \text{ as }n\to\infty.
\end{align}
Thus we get for $n$ large 
\begin{align*}
  \sum_{x_{n,y_2}\leq k<x_{n,y_1}} k^{\alpha-1} e^{-kv_n}
  &\leq 
  \int_{\max\{y_1,e^{-\ell_n}\}}^{e^{v_n}\max\{y_2,e^{-\ell_n}\}}  
  \frac{\big(\ell_n -\log(M) \big)^{\alpha-1}}{(\alpha\log n^*)^{\alpha-1} } \, du
  \leq 
  2\int_{\max\{y_1,e^{-\ell_n}\}}^{e^{v_n}\max\{y_2,e^{-\ell_n}\}}  1 \, du\\
  &\leq 
  2(y_2e^{v_n}  - y_1)
  \leq 
2(y_2e^{v_n} - y_2)  + 2(y_2 - y_1)\\  
&\leq 
2e^{v_n}(y_2 - y_1) +2(y_2 - y_1) 
\leq 5 (y_2-y_1).
\end{align*}
We used on the last line the assumption in \eqref{eq:tightness_criterium} and that $v_n\to 0$.
This completes the proof for $0<\alpha\leq 1$.
For $\alpha>1$, we use that $- \log(u) \leq \ell_n$ in \eqref{eq:thightness_before_alpha} and get
\begin{align*}
  \sum_{x_{n,y_2}\leq k<x_{n,y_1}} k^{\alpha-1} e^{-kv_n}
  &\leq     
  \frac{1}{(\log n^*)^{\alpha-1} } \int_{\max\{y_1,e^{-\ell_n}\}}^{e^{v_n}\max\{y_2,e^{-\ell_n}\}} \big(2\ell_n \big)^{\alpha-1}  \, du
\end{align*}
The remaining computations are the same as for $0<\alpha\leq 1$ and thus this completes the proof.

\subsection{Proof of Theorem~\ref{thm:Longest0}}
\label{sec:proof_longest}
To prove Theorem~\ref{thm:Longest0}, observe that we have for $j\in\N$ and $m\in\N$
\begin{align}
 L_j=
 \max\left\{m;\, \sum_{k= m}^n C_k \geq j \right\}
\end{align}
where $L_j$ is the length of the $j$'th longest cycle.
We now define 
\begin{align}
 \widetilde{L}_j:=
 \max\left\{m;\, \sum_{k= m}^{2n^*\ell_n} C_k \geq j \right\}
\end{align}
We thus immediately get for $y>0$ with the definition of $P_j^{(n)}$ in \eqref{eq:def_Py} that
\begin{align}
 \left\{\frac{\widetilde{L}_{j}-n^*\ell_n}{n^*} < -\log(y)\right\}
 = 
  \left\{\widetilde{L}_{j} < n^*\big(\ell_n-\log(y)\big)\right\}
 =
 \{P_y^{(n)} <j\}.
\end{align}
By Theorem~\ref{thm:Poisson_process}, we get that 
\begin{align}
 \Pb{\frac{\widetilde{L}_{j}-n^*\ell_n}{n^*} < -\log(y)}
 \to
 e^{-y}\sum_{k=0}^{j-1}  \frac{y^k}{k!}.
\end{align}
Recall, the cumulative distribution function of an Exp$(1)$ distributed random variable $E_1$ is $\Pb{E_1\leq u}= 1-e^{-u}$.
We thus get with $j=1$ and $y=e^{-x}$ that 
\begin{align}
\Pb{\frac{\widetilde{L}_{1}-n^*\ell_n}{n^*} < x}
\to
e^{-e^{-x}}
=
\Pb{E_1\geq e^{-x}} 
= 
\Pb{-\log(E_1)\leq x}.
\end{align}
This establishes the asymptotic behaviour of $\widetilde{L}_{1}$.
We could determine the asymptotic behaviour of the vector $(\widetilde{L}_{1}$, \ldots, $\widetilde{L}_{K})$ in similar way as for $\widetilde{L}_{1}$.
However, instead to do this directly, it is easier to us the distributions of the jump times of the Poisson process (see,e.g. \cite[p.5]{Li10}).
Denote by $y_1^{(n)}<y_2^{(n)}<\ldots$ the jump times of the process $P_y^{(n)}$. 
These jumps corresponds to the $\widetilde{L}_{j}$ by the identity $\widetilde{L}_{j}  =  n^*\big(\ell_n-\log(y_j^{(n)})\big)$.
Further, since the process $P_y^{(n)}$ converges to a Poisson process, 
we know that 
\begin{align}
 \big(y_1^{(n)},y_2^{(n)},\ldots, y_K^{(n)} \big)
 \stackrel{d}{\longrightarrow} 
 \left(E_1, E_1+E_2, \ldots, \sum_{j=1}^K E_j \right),
\end{align}
where $(E_j)_{j=1}^K$ is a sequence of iid Exp$(1)$ distributed random variables.
This then implies that 
\[
\frac{1}{n^*}\cdot\left(\widetilde{L}_{1}-n^*\ell_n,\dots,\widetilde{L}_{K}-n^*\ell_n\right)
\stackrel{d}{\longrightarrow}
\left( -\log(E_1), \ldots, -\log\left(\sum_{j=1}^K E_j\right)  \right).
\]
To complete the proof of Theorem~\ref{thm:Longest0}, we consider the event $B_n:= \{\sum_{k>2n^*\ell_n} C_k \geq 1 \}$.
We now get with the Markov inequality
\begin{align}
 \PT{B_n} = \PT{\sum_{k>2n^*\ell_n} C_k \geq 1} \leq \ET{\sum_{k>2n^*\ell_n} C_k}.
\end{align}
Using Corollary~\ref{cor:thm:generalasymptotic}, we get with a similar computation as in Section~\ref{sec:Tightness} that
\begin{align}
 \PT{B_n} 
 \leq 
 2\sum_{k>2n^*\ell_n} \frac{\theta_k}{k} e^{-kv_n}
 =
 O\left(\frac{(2n^*\ell_n)^{\alpha-1} e^{-v_n 2n^*\ell_n}}{v_n} \right)
 = 
 O\left((n^*)^{-\alpha} \log^{1-\alpha}n  \right).
\end{align}
Thus $\PT{B_n} \to 0$. Further $(\widetilde{L}_{1},\ldots,\widetilde{L}_{K})$ and  $(L_{1},\ldots,L_{K})$ agree on the complement of $B_n$.
This completes the proof of Theorem~\ref{thm:Longest0}.

\bibliography{literatur}
\bibliographystyle{abbrv}

\end{document}